\documentclass{svproc}
\usepackage{bm}
\usepackage{amssymb}
\usepackage{amsfonts}
\usepackage{amsmath}
\usepackage{graphicx}
\usepackage{subfigure}      
\usepackage{longtable,tabularx,booktabs}

\usepackage{url}

\begin{document}
\mainmatter              
\title{Generalized  uncertain theory: concepts and fundamental principles}
\titlerunning{Generalized Uncertain Theory}  
%
\author{Xingguang Chen}
\authorrunning{Xingguang Chen} 
%
%
\institute{Jianghan University, Wuhan 430056, P.R.China,\\
\email{cxg@nju.edu.cn}}

\maketitle              

\begin{abstract}
Although there are many mathematical theories to address uncertain phenomena however, these theories are presented under implicit presupposition that uncertainty of objects is accurately measurable while not considering that the measure of uncertainty itself may be inaccurate. Considering this evident but critical overlook, on the basis of reviewing and commenting several widely used mathematical theories of uncertainty, the fundamental concepts and axiomatic system of generalized uncertain theory (GUT)are proposed for the purpose of describing and analyzing that imprecision of objects has inaccurate attributes. We show that current main stream theories of studying uncertain phenomena, such as probability theory, fuzzy mathematics, etc., are the special cases of generalized uncertain theory. So the generalized uncertain theory could cover previous main stream theories of studying uncertainty. Further research directions and possible application realms are discussed. It may be a beneficial endeavor for enriching and developing current uncertainty mathematical theories.
\end{abstract}
\section{Introduction}
The real world is full of uncertainty. There are many events that people can't predict accurately from nature to human society, which has brought great difficulties and challenges to human beings.For example, long-time wether forecasting is very difficult,unexpected geological disasters could make huge casualties, spread of infectious diseases or financial crisis are often unpredictable.Sudden public safety accident, especial the effective prevention of terrorist attacks is becoming a worldwide problem. If we can predict these disasters only before a minor time, the loss of the accident maybe greatly be reduced, while it is very difficult to get the accurate and sufficient information until the accident happen. All these rise a critical question that how to understand the uncertainty properly?
\\\indent
The essential reason of uncertain emerging is because there are many factors affecting the development and change of things around the world, and it is hardly impossible for people to grasp all factors. Moreover, there are also a large number of complex interconnections between components consisting of system , as well as exiting mutual restriction between them. It is difficult for people to fully understand the interaction between these components.
\\\indent
In a sense, certainty is relative while uncertainty is absolute. Uncertainty is an essential attribute of the objective material world. Because of the universal existence of uncertainty in human society and nature, many theoretical tools have been developed to explore it, and the ability of understanding and controlling  uncertainties is constantly enhanced from breadth and depth.
\\\indent
Mathematics could be classed as two categories from the uncertainty perspectives, the first is deterministic mathematics and the other is indeterministic mathematics. The first class include classical pure mathematics (analysis, algebra, geometry, number theory, etc.) as well as applied mathematics (optimization, discrete mathematics, combinatorial mathematics, mathematical physics, biological mathematics, etc.).The second type is mainly used to describe, analyze and deal with various uncertainties. There are some representatives such as fuzzy mathematics \cite{paper_Zadeh_1965}, unascertained mathematics \cite{paper_Wang_1990}, uncertain mathematics \cite{book_Liu_2007}, and the subjective probability theory (the basis of Bayes Statistics,representative scholar is British mathematician Bayes), the objective probability theory (this type probability is based on the frequency, which is originated by the middle of the seventeenth century by French mathematician Pascal, Fermat, Holland mathematician Huygens, Swiss mathematician Bernoulli and so on), grey system theory \cite{paper_Deng_1982}, rough set theory \cite{paper_Pawlak_1982}. Furthermore, all these mentioned uncertain mathematics branches could be classified as two categories with regard to the uncertainty coming from  subjective and objective nature of cognition. The first type involves fuzzy mathematics, unascertained mathematics, uncertain mathematics and subjective probability theory, and the second type involves objective probability theory, grey system theory and rough set theory.
\\\indent All these mentioned theory has a common point that the uncertainty is described by employing a one real number. But in many cases, the uncertainty itself is also uncertain and inaccurate. Using one real numerical scalar to describe the uncertainty maybe far away from the essential characteristics of many situations. To admit that the size of the uncertainty could be accurately known is an ideal approximation of many cases. Indeed,The description, analysis, processing and expression of uncertainty should be considered as two levels, one is the size of the uncertainty can be accurately given, that is, there exists uncertainty for studying objects, while the amount of the uncertainty is deterministic. For example, although it is impossible to predict which surface will appear in advance in random dicing experiments. In ideal case, the probability of each face is 1/6, which is deterministic. The second level is the amount of the uncertainty is also indeterministic. For example, in everyday life, we often say that a thing happens less than 50\%, that means the possibility of something happening is between 0$\sim$50\%; the possibility that something happens is at least 80\%, which indicates that the possibility of something happening is between  80\%$\sim$100\%, and so on. From this point of view that the degree of uncertainty is uncertain, the concept of generalized uncertainty measure is developed, which the uncertainty itself and degree of uncertainty of the studying objects can be described at the same time. On this basis, a generalized uncertain axiom system is proposed, and several basic results are taken out. An example of uncertain decision-making is carried out to demonstrated the effectiveness and rationality of GUT.

\section{Fundamental Concepts}
\begin{definition}[Generalized uncertain measure,GUM]
 Let $\Omega$ be a nonempty set, and let $\Gamma$ be a $\sigma$-algebra over $\Omega$. For every element $e_i$, where $e_i \in \Gamma$, if $G$ satisfy three conditions as below:
 \\
(1)~(Generalized nonnegativity) $G(e_i)=[a_i,b_i]\subseteq[0,1]$;
\\
(2)~(Generalized normalization) $G(\Omega)=1$;
\\
(3)~(Generalized countable additivity ) For arbitrary sequences ${e_1,\cdots,e_n}$,we have $G(e_1\cup \cdots\cup e_n)\leq \sum_{i=1}^{i=n}G(e_i):=[\sum{a_i},\sum{b_i}]\subseteq[0,1]$. Specially, if $\{e_i\}$ is incompatible, i.e., $e_i \cap e_j=\varnothing$, then we have $G(e_1\cup \cdots\cup e_n)= \sum_{i=1}^{i=n}G(e_i):=[\sum{a_i},\sum{b_i}]\subseteq[0,1]$, here the number of set element $n$ could be finite or infinite.
\\
Then $G$ is said to be generalized uncertain measure and triples $\{\Omega,\Gamma,G\}$ are its corresponding generalized uncertain space, here $a_i,b_i \in [0,1]$.
\end{definition}
\indent For event $e$, it's generalized uncertain measure is denoted as $G(e)=[a,b]$, if $a\leq b$,then the left point $a$ indicates the minimum value of event being true, and the right point $b$ indicates the maximum value of event being true. From the definition of  generalized uncertain measure, if and only if $a=b$, the probability of even is deterministic. When $\Omega$ is random sampling set, generalized uncertain measure is equivalent to probability measure. When $\Omega$ is fuzzy set, then generalized uncertain measure degenerates to degree of membership of fuzzy theory. When $\Omega$ is uncertain set, then generalized uncertain measure will be uncertain measure derived by uncertain theory.

\begin{definition}[Inverse interval]
 $[a,b]$ is ordinary interval on real number space, i.e., $a,b \in \bbbr$,then $[b,a]$ is the inverse interval of $[a,b]$.
\end{definition}
\begin{definition}[GUM of complementary set]
Assume the generalized uncertain measure of set $A$ is $G(A)=[a,b]$,then the generalized uncertain measure of complementary set of $A$ is $G(A^c)=[1-a,1-b]$.Sometimes the complementary set of $A$ is denoted as $\overline{A}$.
\end{definition}
\begin{definition}[Arithmetic operation of GUM]
 Let generalized  uncertain measure of two set $e_1,e_2$ be $G(e_1)=[a_1,b_1]\subseteq [0,1]$,$G(e_2)=[a_2,b_2]\subseteq [0,1]$ respectively. Then (1)~$G(e_1)+G(e_2)=[a_1+a_2,b_1+b_2]$; (2)~$G(e_1)-G(e_2)=[a_1-a_2,b_1-b_2]$;(3)~$G(e_1)\times G(e_2)=[a_1\cdot a_2,b_1 \cdot b_2]$;(4)~$G(e_1)\div G(e_2)=[a_1/ a_2,b_1/b_2]$,  where $b_1\neq 0$ and $b_2 \neq 0$.
\end{definition}
\begin{definition}[Generalized  uncertain independent set]
We said set $A$ and $B$ is independent, if $G(AB)=G(A)G(B)$.
\end{definition}
\begin{definition}[The comparison of GUM]
 Assume two interval $I_1=[a_1,b_1]$ and $I_2=[a_2,b_2]$, where $a_1 \leq b_1$ and $a_2 \leq b_2$. Then there exist three relations between $I_1$ and $I_2$. i.e., (1)~separation relationship (R1): if $b_1<a_2$, denoted as $I_1 \ll I_2$; (2)~interlaced relationship (R2): if $a_1 \leq a_2 \leq b_1 \leq b_2$, denoted as $I_1 \asymp I_2$; (3)~inclusion relationship (R3): if $a_1 > a_2$ and $b_1 < b_2$, denoted as $I_1 \subseteq I_2$.
 \\ Let generalized  uncertain measure of two set $e_1,e_2$ be $G(e_1)=[a_1,b_1]:=I_1\subseteq [0,1]$,$G(e_2)=[a_2,b_2]:=I_2\subseteq [0,1]$ respectively.(i)~If and only if $I_1$ and $I_2$ satisfies relation R1, it is said that generalized  uncertain measure of set $e_1$ is strongly smaller than $e_2$, denoted as $G(e_1)<G(e_2)$; (ii)~If and only if $I_1$ and $I_2$ satisfies relation R2,it is said that generalized  uncertain measure of set $e_1$ is weakly smaller than $e_2$, denoted as $G(e_1) \leq G(e_2)$; and (iii)~if and only if $I_1$ and $I_2$ satisfies relation R3, it is said that generalized  uncertain measure of set $e_1$ is  partly smaller than $e_2$, denoted as $G(e_1) \preceq G(e_2)$.On the contrary, it is said that strongly greater,  weakly greater, and partly greater, which are denoted as $>$, $\geq$ and $\succeq$ respectively.
\end{definition}
\begin{definition}[The uncertainty degree of GUM]
 It is said that the length of interval of GUM to be taken as the uncertainty degree of GUM of set $e$, which is denoted as $\text{gud}(e)=[b-a]$, where $G(e)=[a,b]$ is GUM of set $e$.
\end{definition}
\begin{definition}[Low order uncertain system and high order uncertain system]
For two uncertain system $S_1$ and $S_2$, if GUM of $S_2$ is greater than GUM of $S_1$, i.e., $G(S_2)>G(S_1)$,then we said $S_2$ is  a higher order uncertain system than $S_1$.
\end{definition}
\begin{definition}[Generalized  uncertain function and  variable]
Let triples $(\Omega, \Gamma, G)$ be generalized  uncertain space, if for arbitrary element $\xi \in \Gamma$, there exists a function family $\textbf{f}(\xi)=\{f(\xi):f_1(\xi)\leq f(\xi) \leq f_2(\xi)\}$,where $f(\xi)$ is ordinary function which is defined on $\Gamma$, and value domain is $[0,1]$, i.e., $f(\cdot): \Gamma \mapsto [0,1]$, then function family $\textbf{f}(\xi)$ is defined as generalized  uncertain function (GUF)and $f(\cdot)$ is core function of  generalized  uncertain function. $f_1(\cdot)$ and $f_2(\cdot)$ are taken account as lower core function and upper core function of $\textbf{f}(\cdot)$ respectively, $\xi$ is said to be generalized  uncertain variable.
\end{definition}
\begin{definition}[$\delta$ neighbour between generalized uncertain variables]
 For non-negative real number $\delta \geq 0$, there are two generalized uncertain variables $\xi_1$ and $\xi_2$. Assume their GUMs are $G(\xi_1)=[a_1,b_1]$ and $G(\xi_2)=[a_2,b_2]$ respectively. We assert generalized uncertain variable $\xi_1$ and $\xi_2$ are $\delta$ neighbour each other, if $\xi_1$ and $\xi_2$ satisfy: $\mid a_1-a_2 \mid\leq \delta$ and $\mid b_1-b_2\mid\leq \delta$.
\end{definition}
\begin{definition}[Generalized  uncertain distribution function]
The generalized  uncertain distribution function are defined as $G_d(x)=G(\xi\leq x)$, where $G(\cdot)$ is generalized  uncertain function.
\end{definition}
\begin{definition}[Generalized uncertain density function, GUDF]
 We call function family $\textbf{f}_{\xi}(s)$ is generalized uncertain density function, if  function family $\textbf{f}_{\xi}(s)$ satisfy that $G_d(x)=G(\xi\leq x)=\int_{-\infty}^{x}\textbf{f}_{\xi}(s)ds:=[\min\{\int_{-\infty}^{x}f_{\xi}(s)ds\},\max\{\int_{-\infty}^{x}f_{\xi}(s)ds\}]\subseteq [0,1]$.
\end{definition}
\begin{definition}[Generalized uncertain mass function, GUMF]
 For discrete  generalized uncertain variables $\xi\in\{x:x=x_1,\cdots,x_n\}$, we call $G_M(x)=G(\xi=x)$ is generalized uncertain mass function, the  distribution law of discrete  generalized uncertain variables is denoted as $\{G(\xi=x_i)\}=\{[a_i,b_i]\}$.
\end{definition}
\begin{definition}[Generalized uncertain expectation, GUE]
(1)GUE of discrete generalized uncertain variables: Assume the distribution law of discrete  generalized uncertain variable $\xi$ is $\{G(\xi=x_i)\}=\{[a_i,b_i]\}$, where $i=1,\cdots, n$, then the GUE is $GUE(\xi)=[\sum_{i=1}^{n}x_ia_i, \sum_{i=1}^{n}x_ib_i]$;
\\ (2)GUE of continue generalized uncertain variables: $GUE(\xi)=\int_{-\infty}^{\infty}x\cdot \textbf{f}_{\xi}(x)dx:=[\min\{\int_{-\infty}^{\infty}x\cdot f_{\xi}(x)dx\}, \max\{\int_{-\infty}^{\infty}x\cdot f_{\xi}(x)dx\}]$.
\end{definition}
\begin{definition}[Generalized covariance of generalized uncertain variables]
Assume $(\xi_1, \xi_2)$ is two dimension generalized uncertain variables, if  $GUE\{(\xi_1-GUE(\xi_1))(\xi_2-GUE(\xi_2))\}$ exists, then we call it is the generalized covariance of generalized uncertain variables $\xi_1$ and $\xi_2$.
\end{definition}
\begin{definition}[Generalized uncertain process]
 $(\Omega, \Gamma, G)$ is a generalized uncertain space, $T$ is parameter set, where $T\subset \bbbr$. For every $t \in T$, there exists one generalized uncertain variable $\xi(\omega, t)$, then set $\{\xi(\omega, t)\}$ are taken as generalized uncertain process defined on $(\Omega, \Gamma, G)$. It could be denoted as $\{\xi(\omega,t); \omega \in \Omega, t\in T\}$ or $\{\xi(t);t\in T\}$, abbreviated as $\{\xi(t)\}$. All possible value space $S$ of generalized uncertain variable $\xi$ taken on parameter set $T$ is said to be status of generalized uncertain process. Specially, if the parameter set $T$ is discrete, then generalized uncertain process $\{\xi(t)\}$ is said to be as generalized uncertain  sequences.
\end{definition}
\begin{definition}[Generalized uncertain limit, variation, derivative, and integral]
Assume generalized  uncertain function $\textbf{f}(\xi)=\{f(\xi):f_1(\xi)\leq f(\xi) \leq f_2(\xi)\}$,
\\(i) Generalized uncertain limit:
$\lim \limits_{\xi \to {\xi_0}}\textbf{f}(\xi)=[\min\{\lim \limits_{\xi \to {\xi_0}}f(\xi)\}, \max\{\lim \limits_{\xi \to {\xi_0}}f(\xi)\}]$,where  $f(\xi)$ is  core function of GUF $\textbf{f}(\xi)$;
\\(ii) Generalized uncertain derivative: $\textbf{f}_{x=x_0}^{'}=\lim \limits_{x \to {x_0}}\frac{\Delta \textbf{f}}{\Delta {x}}=[\min\{\lim \limits_{x\rightarrow {x_0}}\frac{\Delta {f}}{\Delta {x}}\},\max\{\lim \limits_{x\rightarrow {x_0}}\frac{\Delta {f}}{\Delta {x}}\}]$;
\\(iii) Generalized uncertain variation: $\delta \textbf{f}=\textbf{f}(x+\Delta x)-\textbf{f}(x)=[\min\{f(x+\Delta x)-f(x)\}, \max\{f(x+\Delta x)-f(x)\}]$;
\\(iv) Generalized uncertain integral: $\int_a^b \textbf{f}(x)dx =[\min\{\int_a^b {f}(x)dx\},\max\{\int_a^b {f}(x)dx\}].$.
\end{definition}

\section{Corollaries, Propositions and Algorithms}
\noindent
\begin{corollary}[Conditional generalized uncertain measure] For set $A$ and $B$,we have $G(A\vert B)=\dfrac{G(AB)}{G(B)}$.
\end{corollary}

\begin{proof} Assume the total number of experiment is $n$. The number of event A occurs is $[\underline{n}_A, \overline{n}_A]$, which indicates the minimum number and maximum number of event A happening is $\underline{n}_A$ and  $\overline{n}_A$ respectively.
It is an uncertain range from $\underline{n}_A$ to $\overline{n}_A$. Similarly,the number of event B occurs is $[\underline{n}_B, \overline{n}_B]$, and the number of event A and event B occur synchronously is $[\underline{n}_{AB}, \overline{n}_{AB}]$. Note, there is only the total number of experiment is determined. The conditional generalized uncertain measure $G(A \vert B)$ should be the number of  event A and event B occur synchronously is divided by   the number of  event B occurs, i.e., $G(A \vert B)=\dfrac{[\underline{n}_{AB}, \overline{n}_{AB}]}{[\underline{n}_B, \overline{n}_B]}=\dfrac{[\underline{n}_{AB}, \overline{n}_{AB}]/n}{[\underline{n}_B, \overline{n}_B]/n}=\dfrac{G(AB)}{G(B)}$.
\qed
\end{proof}

\begin{corollary} If the generalized uncertain degree $gud(A)=0$,then the generalized uncertain measure of A is a real number.
\end{corollary}
\begin{proof} From the definition of GUM, $gud(A)=l(G(A))=b-a=0\Rightarrow a=b$, then we have $G(A)=[a,b]=[a,a]=a$, which means the assertion is true. \qed
\end{proof}

\begin{corollary}[Monotonicity of generalized uncertain measure] For arbitrary two generalized uncertain measurable sets $A_1$ and $A_2$, if $A_1\subseteq A_2$,then the GUM of these two sets maybe two situations, ~(i)$G(A_1)<G(A_2)$, ~(ii)$G(A_1)\leq G(A_2)$, while the third situation, i.e., partly smaller won't be came into existence.
\end{corollary}
\begin{proof} Assume $G(A_1)=[a_1,b_1]$,$G(A_2)=[a_2,b_2]$,$A_3:=A_2-A_1=A_2/A_1$,$G(A_3)=[a_3,b_3]$, then $A_1\cap A_3=\varnothing$. We have $G(A_2)=G(A_1 \cup A_3)=G(A_1)+G(A_3)=[a_1+a_3, b_1+b_3]$,then because $a_1 \in [0,1]$,$b_1 \in [0,1]$, $a_3 \in [0,1]$,$b_3 \in [0,1]$, and it is obvious that $0\leq a_1 \leq a_1+a_3 \leq 1$ and $0\leq b_1 \leq b_1+b_3 \leq 1$, (i) if $b_1 < a_1+a_3$, then $[a_1,b_1]$ and $[a_1+a_3, b_1+b_3]$ satisfy separation relationship (R1),which means $G(A_1)<G(A_2)$; (ii) if $b_1 \geq a_1+a_3$,then $[a_1,b_1]$ and $[a_1+a_3, b_1+b_3]$  satisfy interlaced relationship (R2), which means $G(A_1)\leq G(A_2)$. It is obvious that the third situation won't occur. The proof is completed. \qed
\end{proof}
\begin{corollary}[The additional formula of GUM]   For arbitrary two generalized uncertain measurable sets $A_1$ and $A_2$, we have $G(A_1 \cup A_2)=G(A_1)+G(A_2)-G(A_1A_2)$.
\end{corollary}
\begin{proof} $A_1A_2 \subseteq A_2 \Rightarrow A_2=(A_2-A_1A_2)\cup(A_1A_2)$, and $(A_2-A_1A_2)\cap (A_1A_2)=\varnothing$.So, $G(A_2)=G(A_2-A_1A_2)+G(A_1A_2)$, then $G(A_2-A_1A_2)=G(A_2)-G(A_1 A_2)$. Furthermore, $A_1 \cup A_2=A_1 \cup (A_2-A_1A_2)$, and $A_1 \cap (A_2-A_1A_2)=\varnothing$, so we have $G(A_1\cup A_2)=G(A_1)+G(A_2-A_1A_2)=G(A_1)+G(A_2)-G(A_1 A_2)$.  \qed
\end{proof}
\begin{proposition}
 For generalized uncertain set sequences $A=\{A_i;i=1\cdots n\}$, the GUM for every $A_i$ is $G(A_i)=[a_i,b_i]$, and  arbitrary two sets are incompatible, i.e., $A_i\cap A_j= \varnothing $. If the GUD of every $A_i$ equals 0, i.e., $\text{gud}(A_i)=0$,then we have $G(A)=\sum_{i=1}^{i=n}a_i$.
\end{proposition}
\begin{proof} From Corollary 2 we have $G(A_i)=a_i \in \bbbr$, furthermore, considering condition 3 of definition 1, we have $G(A)=G(A_1 \cup A_2 \cdots \cup A_n)=\sum_{i=1}^{i=n}G(A_i)=\sum_{i=1}^{i=n}a_i$.
\qed
\end{proof}
\begin{proposition}[The correlation between probability and GUM]  In independent repeated experiment, let $n$ be the number of experiment performed. $P(A)$ is the probability that the event occurs in one stochastic experiment. The GUM of $n$th stochastic experiment is recorded as $G(A)_n=[a_n,b_n]$,and $\{G(A)_n\}=\{[a_n,b_n]\}$ is consist of  a closed nested intervals. Then we have $P(A)=\lim \limits_{n\rightarrow \infty}G(A)_n$, i.e., the probability is the limitation of generalized uncertain measure under this specific conditions.
\end{proposition}
\begin{proof} From the closed nested interval theorem, there exists only one real number $\xi \in \bbbr$, which satisfy $\lim \limits_{n\rightarrow \infty} a_n =\lim \limits_{n\rightarrow \infty} b_n =\xi$, from the definition of probability, we have $P(A)=\lim \limits_{n\rightarrow \infty} a_n=\xi=[\xi,\xi]=[\lim \limits_{n\rightarrow \infty} a_n, \lim \limits_{n\rightarrow \infty} b_n]=\lim \limits_{n\rightarrow \infty}[a_n,b_n]=\lim \limits_{n\rightarrow \infty} G(A)_n$. \qed
\end{proof}
\begin{proposition}[The correlation between degree of membership and GUM] Assume $\mu(A)$ is the degree of membership function of set $A$ under fuzzy theory, $G(A)=[a,b]$ is the GUM of set $A$. Then $\mu(A)$ is a special case of GUM.
\end{proposition}
\begin{proof} Make $a=b=\mu(A)$, then $\mu(A)=a=b=[a,b]=G(A)$, which indicates that degree of membership is equivalent to a special case that  the upper and lower uncertain value are equal to the same value.  \qed
\end{proof}
\begin{solution}[Generation of generalized uncertain  sequences]
 A generalized uncertain sequence including $k$ elements could be produced by random variables having arbitrary distribution. Assume there are random variables $X_1,\cdots,X_n$. Their distributions  are $X_1\thicksim f_1(\mu_1, \sigma_1^2),\cdots, X_n\thicksim f_n(\mu_n,\sigma_n^2)$, where the $\mu_i, \sigma_i^2$ is mean and variance of  random variable $X_i$. Below is brief introduction of algorithm processing.
\\Step 1. Produce $k$ random number sequens by employing random number generator. The $j$th sequence is marked as $\mathbf{\eta}_j=\{X_{1j},\cdots,X_{nj}\}$, the $i$th element obeys $i$th distribution, i.e., $X_{ij}\thicksim f_i(\mu_i,\sigma_i^2)$, let $j=1$.
\\Step 2. Generate a integer $1\leqslant l_j \leqslant k$, where $l_j$ obeys uniform distribution between $[1,k]$. Then we select the $l_j$th elements which belongs to $j$th random number sequence $\eta_j$. It is marked as $X_{{l_j}j}$. Then the mean of  $X_{{l_j}j}$ is between $[\min{\mu_j},\max{\mu_j}]$. So $X_{{l_j}j}$ could be taken as $j$th generalized uncertain number;
\\Step 3. Let $j=j+1$, if $j<k+1$, then return to Step 2, continue executing it. Otherwise, after totally $k$ times operation, we can get an approximate generalized uncertain sequences $\xi=\{X_{{l_1}1}, \cdots, X_{{l_k}k}\}$, where $X_{{l_j}j}$ is the generalized uncertain number which is generated at $j$th step.
\end{solution}
\begin{solution}[Plain fast classing algorithm based on $\delta$ neighbour]
 For $\delta >0$, and generalized uncertain sequences $A=\{\xi_1, \cdots \xi_n\}$ , the plain fast classing algorithm based on $\delta$ neighbour is articulated as below:
\\Step 1. Set  integer $k=1$,  initial classing set $B_k=\{\xi_1\}$, object  set $C=A$;
\\Setp 2. Go through object set $C$ according to this process: take arbitrary element $\xi_i \in C, i \neq 1$,if $\xi_i$ and $\xi_1$ is $\delta$ neighbour, then add $\xi_i$ into set $B$, update $B=\{\xi_1, \xi_i\}:=\{\xi^k_1,\xi^k_2\}$. After one time ergodicity, the classing set  $B_k=\{\xi^k_1, \cdots \xi^k_{m_k} \}$;
\\Setp 3. $k=k+1$, update object  set $C=A-(B_1 \cup B_2 \cup \cdots \cup B_{k-1})$, if object set $C=\varnothing$, then we get set sequences $\{B_1,\cdots, B_{k-1}\}$, classing is completed. Otherwise, we continue going through object  set $C$ to get new  set  $B_k$, then return Step 3 to repeat this step.
\end{solution}

\section{Applications of GUT on decision making}
There are two kinds uncertain decision types, the one is uncertain type decision and the other is  risk  type decision. The feature of the first one includes (1)The states of nature is already known by decision maker, (2) The revenue under different nature is already known; and (3) The status of nature isn't determined and its probability distribution isn't known in advance. For classical uncertain type decision problems, there are five common decision criteria could be employed in theoretical  or practical research, involving (1) pessimistic criteria; (2) optimistic criteria; (3) compromise criteria; (4) minimum maximum regret criterion and (5) equal possibility criterion. These five decision criteria could be described uniformly by employing GUT. Furthermore, classical uncertain type and risk type decision problems could be addressed by adopting GUT as well.
\subsection{ Principles of generalized uncertain decision}
We assume the total number of possible natural status is $n$ and it is denoted as $\{N_i:i=1,\cdots,n\}$. Its corresponding generalized uncertain distribution is $G(N_i)=G(\xi=N_i)=[a_i,b_i]$,where $i\in\{1,\cdots,n\}$. Schemes set is $\{S_i;i=1,\cdots, m\}$. Payoff matrix is $A=\{a_{i,j}\}$, where $i\in\{1,\cdots,m\},j\in\{1,\cdots,n\}$. Here $a_{i,j}$ denotes the actor's payoff under $j$th natural condition when adopting $i$th scheme. The generalized expected utility adopting scheme $S_i$ could be formulated as $GEU(S_i)=\sum_{j=1}^{n}a_{i,j}G(N_j)$. When $G(N_j)=[p_1,p_2]$ and $p_1<p_2$, it is uncertain type decision in the classical uncertain decision making context, when $G(N_j)=[p_j,p_j]=p_j$, here $p_j$ indicates the probability of natural status $N_j$ occurs, then it could be taken as risk type decision in the classical uncertain decision making context. Corresponding decision criteria could be summarized as below:~Firstly, select the strongly advantage scheme among $\{GEU(S_i)\}$; Secondly, if there doesn't exist strongly advantage scheme, then select weakly advantage  scheme among $\{GEU(S_i)\}$; Thirdly, if the strongly advantage scheme and weakly advantage scheme are both absent, which means inclusion relationship occurs between some schemes, then we make decision according below criteria: (i) If the decision maker is robust (risk aversion), we should select the scheme which has the smallest generalized uncertainty degree of GUM, i.e., $S_k=\min \{gud(GEU(S_i))\}$, where $S_k$ is selected scheme; (ii) If the decision maker is radical (risk seeking), we should select the scheme which has the greatest generalized uncertainty degree of GUM, i.e., $S_k=\max \{gud(GEU(S_i))\}$. Below we'll elaborate this process by a brief decision example.

\subsection{ Examples of generalized uncertain decision}
\begin{example}
Assume there are totally four schemes to be selected, and natural circumstance has three distinct status respectively. The generalized distribution of natural circumstance, payoff matrix and scheme's GEU are listed in Table 1.
\begin{table}[htbp]
\begin{center}
\caption{Payoff matrix and GEU of four schemes}
 \begin{tabular}{cccccc}
\toprule[1pt]
{/} & {Status 1} & {Status 2} & {Status 3} & {GEU} &{Comparison}
\\
\midrule
GUM     &  [0.1,0.2]  & [0.2,0.3] & [0.5,0.7] & /  & /\\
S1     &  100     & 80   & 90     & [71,107] & /\\
S2     &  120     & 130  & 110    & [93,140]  & GEU2 $\geq$ GEU1\\
S3     &  150     & 150  & 120    & [105,159] & GEU3 $\geq$ GEU2\\
S4     &  160     & 90   & 140    & [104,157] & GEU4 $\leq$ GEU3\\
\bottomrule[1pt]
\end{tabular}
\end{center}
\end{table}
It is obvious that the GEU of scheme3 is the most weakly advantage among four GEUs, so the scheme3 could be selected as our final choice in this scenario.
If there is a new scheme which is noted as scheme5 added in scheme set, its payoff under three natural conditions are 0, 530 and 0 respectively. Similar as Table 1, the generalized distribution of natural circumstance, payoff matrix and scheme's GEU are listed in Table 2.
\begin{table}[htbp]
\begin{center}
\caption{Payoff matrix and GEU of five schemes}
 \begin{tabular}{cccccc}
\toprule[1pt]
{/} & {Status 1} & {Status 2} & {Status 3} & {GEU} &{Comparison}
\\
\midrule
GUM     &  [0.1,0.2]  & [0.2,0.3] & [0.5,0.7] & /  & /\\
S1     &  100     & 80   & 90     & [71,107] & /\\
S2     &  120     & 130  & 110    & [93,140]  & GEU2 $\geq$ GEU1\\
S3     &  150     & 150  & 120    & [105,159] & GEU3 $\geq$ GEU2\\
S4     &  160     & 90   & 140    & [104,157] & GEU4 $\leq$ GEU3\\
S5     &  0       & 530   & 0    &  [106,159]  & GEU5 $\preceq$ GEU3\\
\bottomrule[1pt]
\end{tabular}
\end{center}
\end{table}
First, from above analysis, we know that the third scheme, i.e., scheme3 is the optimal choice among scheme1 to scheme4. Second, because the GEU3 is partly smaller than GEU5, so according to the  decision criteria presented in previous subsection, if the decision maker is risk aversion style,we should select scheme5, while if the decision maker is risk seeking style, we should select scheme3.
\end{example}

\section{Concluding and Discussions}
Although there are numerous mathematical tools to address uncertain phenomena, and whatever which theory is employed, the final attempt is to get one real number to describe the uncertainty accurately. An obvious but important overlook lies in that the measure of uncertainty itself may be inaccurate. Traditional processing approach that using a precise number to measure  uncertainty could be an ideal assumption. In order to overcome these limitations, we develop a whole set of fundamental framework from theoretical  perspective by reviewing some existing  theory in uncertainty domains. Related concepts and axiomatic system of generalized uncertain theory such as generalized uncertain measure (GUM), arithmetic operation of GUM,generalized  uncertain function and variable,generalized  uncertain distribution, generalized uncertain process,  generalized uncertain limit, generalized uncertain variation, and generalized uncertain derivative, etc., are derived in this framework. Some corollaries, propositions and algorithms are presented as well to extend the application scope of the theory. Furthermore, we give an elaboration about how to use the GUT to address complex uncertain and risk type decision making problems, and an example is carried out to demonstrated the effectiveness and rationality about our proposed theory. Being a novel mathematical theory, we have just started a little step, many  theoretical or applications issues need to be further explored of course. Extending and incorporating  the GUT into other realms, such as systems recognition,forecasting, optimization and control ,system evaluation, and system decomposition, etc.,are all future possible directions.

\paragraph{Acknowledgments}
This work is supported by the National Natural Science Foundation of China (Grant No. 71471084).

%
%


\begin{thebibliography}{10}
%

\bibitem{paper_Deng_1982}
Deng, J.L.
\newblock (1982)~Control problems of grey systems.
\newblock {\em Systems \& Control Letters}, 5(1):288--294.


\bibitem{book_Liu_2007}
Liu, B.D.
\newblock {\em  Uncertainty Theory}.
\newblock Springer-Verlag Berlin Heidelberg, 2nd edition, 2007.


\bibitem{paper_Pawlak_1982}
Pawlak,Z.
\newblock (1982)~Rough sets.
\newblock {\em International Journal of Parallel Programming}, 11(5):341--356.

\bibitem{paper_Wang_1990}
Wang, G.Y.
\newblock (1990)~Unascertained information and its mathematical treatment.
\newblock {\em Journal of Harbin University of Civil Engineering}~(in Chinese), 23(4):1--9.


\bibitem{paper_Zadeh_1965}
Zadeh, L.
\newblock (1965)~ Fuzzy sets.
\newblock {\em Information and Control}, 8(3):338--353.



\end{thebibliography}
\end{document}